\documentclass{amsart}
\usepackage{amssymb}
\usepackage{amsmath,amsthm}
\usepackage{graphicx}
\title[Small exotic rational surfaces without 1- and 3-handles]{Small exotic rational surfaces\\without 1- and 3-handles}
\author[Kouichi Yasui]{Kouichi Yasui}
\date{July 23, 2008}
\address{Department~of~Mathematics, Graduate~School~of~Science, Osaka~University, Toyonaka, Osaka 560-0043, Japan}
\email{kyasui@cr.math.sci.osaka-u.ac.jp}
\thanks{The author is partially supported by JSPS Research Fellowships for Young Scientists.}

\newtheorem{theorem}{Theorem}[section]
\newtheorem{proposition}[theorem]{Proposition}
\newtheorem{lemma}[theorem]{Lemma}
\newtheorem{corollary}[theorem]{Corollary}
\theoremstyle{definition}
\newtheorem{definition}[theorem]{Definition}
\newtheorem{remark}[theorem]{Remark}
\newtheorem{question}[theorem]{Question}
\newtheorem{ack}{Acknowledgement}

\begin{document}

\begin{abstract}
We give new rational blowdown constructions of exotic $\mathbf{CP}^2\# n\overline{\mathbf{C}\mathbf{P}^2}$ $(5\leq n\leq 9)$ without using elliptic fibrations. We also show that our 4-manifolds admit handle decompositions without 1- and 3-handles, for $7\leq n\leq 9$. A strategy for rational blowdown constructions of exotic $\mathbf{CP}^2\# n\overline{\mathbf{C}\mathbf{P}^2}$ $(1\leq n\leq 4)$ is also proposed.
\end{abstract}

\maketitle

\section{Introduction}
It is not known if every smooth $4$-manifold admits an exotic smooth structure. Various methods for constructing exotic smooth structures on small $4$-manifolds are currently in rapid progress. Park~\cite{P1}, Stipsicz-Szab\'{o}~\cite{SS}, Fintushel-Stern~\cite{FS2} and Park-Stipsicz-Szab\'{o}~\cite{PSS} constructed exotic smooth structures on $\mathbf{CP}^2\# n\overline{\mathbf{C}\mathbf{P}^2}$ $(5\leq n\leq 8)$ by using rational blowdowns as main tools. They used elliptic fibrations (and knot surgeries) to perform rational blowdowns. Akhmedov-Park~\cite{AP}, \cite{AP2} \textit{et.\ al.\ }later constructed exotic $\mathbf{CP}^2\# n\overline{\mathbf{C}\mathbf{P}^2}$ $(2\leq n\leq 4)$ without using rational blowdowns. 

However, it is still unknown whether or not $S^4$ and $\mathbf{CP}^2$ admit an exotic smooth structure. If such a structure exists, then every handle decomposition of it must contain 1- or 3-handles (cf.\ \cite{Y2}). 
On the contrary, many classical closed $4$-manifolds are known to admit neither $1$- nor $3$-handles (cf.~\cite{GS}, \cite{A1}, \cite{Y3}). 
Thus the following question seems to be reasonable: ``What is the smallest $n$ for which an exotic $\mathbf{CP}^2\# n\overline{\mathbf{C}\mathbf{P}^2}$ without $1$- and $3$-handles exists?''. In \cite{Y2}, we constructed an exotic $\mathbf{CP}^2\# 9\overline{\mathbf{C}\mathbf{P}^2}$ without $1$- and $3$-handles by using rational blowdowns and Kirby calculus. In \cite{A1}, Akbulut later proved that the elliptic surface $E(1)_{2,3}$, which is an exotic $\mathbf{CP}^2\# 9\overline{\mathbf{C}\mathbf{P}^2}$, has neither 1- nor 3-handles by using knot surgery and investigating a dual handle decomposition of $E(1)_{2,3}$. 

The purpose of this paper is two-fold. The first is to give new constructions of exotic $\mathbf{CP}^2\# n\overline{\mathbf{C}\mathbf{P}^2}$ $(5\leq n\leq 9)$ by using rational blowdowns, Kirby calculus and no elliptic fibrations. Our constructions give explicit procedures to draw handlebody pictures. In particular, our manifolds are the first examples in the following sense: 
\begin{theorem}
$(1)$ For $7\leq n\leq 9$, there exists a smooth $4$-manifold which is homeomorphic but not diffeomorphic to $\mathbf{CP}^2\# n\overline{\mathbf{C}\mathbf{P}^2}$ and has neither $1$- nor $3$-handles in a handle decomposition. \\
$(2)$ There exists a smooth $4$-manifold which is homeomorphic but not diffeomorphic to $\mathbf{CP}^2\# 6\overline{\mathbf{C}\mathbf{P}^2}$ and has no $1$-handles in a handle decomposition.
\end{theorem}
In general, it is difficult to show that exotic 4-manifolds admit neither 1- nor 3-handles. See, for example, \cite{A1}. However, in \cite{Y2} and this paper, we constructed our exotic rational surfaces so that their $1$- and $3$-handles naturally disappear. Thus it is easy to elliminate 1- and 3-handles of our handlebodies. 

The second purpose is to propose a strategy for rational blowdown constructions of exotic $\mathbf{CP}^2\# n\overline{\mathbf{C}\mathbf{P}^2}$ $(1\leq n\leq 9)$, though the author could not carry out the strategy for $1\leq n\leq 4$. The author thought of the strategy in connection with a natural question on handle decompositions of $\mathbf{CP}^2\# 2\overline{\mathbf{C}\mathbf{P}^2}$. 
\begin{ack}
The author would like to thank his adviser Hisaaki Endo and Kazunori Kikuchi for their heartfelt encouragements and discussions. 
This work is based on a part of the author's master thesis in 2006.
\end{ack}
\section{Rational blowdown}
In this section we review the rational blowdown introduced 
by Fintushel-Stern \cite{FS1}. For the procedure to draw handlebody diagrams of rational blowdowns, see also Gompf-Stipsicz \cite[Section~8.5]{GS}.

Let $C_p$ and $B_p$ be the smooth $4$-manifolds defined by handlebody diagrams in Figure~\ref{C_p}, and $u_1,\dots,u_{p-1}$ elements of $H_2(C_p;\mathbf{Z})$ given by corresponding $2$-handles in the figure such that $u_i\cdot u_{i+1}=+1$ $(1\leq i\leq p-2)$.
The boundary $\partial C_p$ of $C_p$ is diffeomorphic to the lens space $L(p^2,p-1)$ and to the boundary $\partial B_p$ of $B_p$. 
\begin{figure}[htbp]
\begin{center}
\includegraphics[width=3.5in]{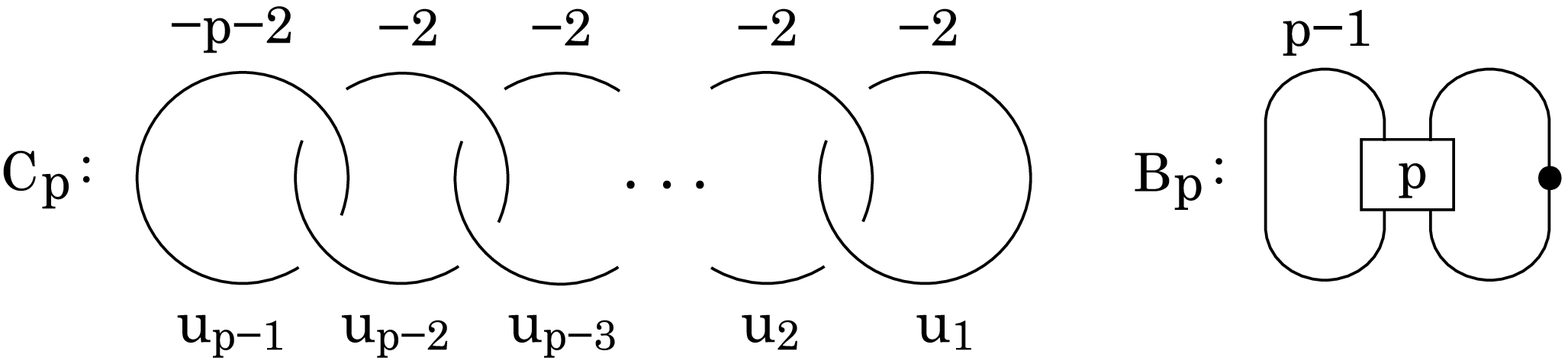}
\caption{}
\label{C_p}
\end{center}
\end{figure}

Suppose that $C_p$ embeds in a smooth $4$--manifold $X$. 
Let $X_{(p)}$ be the smooth $4$-manifold obtained from $X$ by removing $C_p$ and gluing $B_p$ along the boundary. The smooth $4$-manifold $X_{(p)}$ is called the rational blowdown of $X$ along $C_p$. Note that $X_{(p)}$ is uniquely determined up to diffeomorphism by a fixed pair $(X,C_p)$ (see Fintushel-Stern~\cite{FS1}). 
This operation preserves $b_2^+$, decreases $b_2^-$, may create torsions in the first homology group.
\section{Construction}\label{section:construction}
In this section we give constructions of exotic $\mathbf{CP}^2\# n\overline{\mathbf{C}\mathbf{P}^2}$ $(5\leq n\leq 9)$. In handlebody diagrams, we write the second homology classes given by $2$-handles, instead of usual framings. Note that the square of the homology class given by a 2-handle is equal to the usual framing. We do not draw (whole) handlebody diagrams of exotic rational surfaces and the other manifolds appeared in the following construction. However, one can draw whole diagrams, following the procedures in this section. 

Let $h,e_1,e_2,\dots,e_n$ be a canonical orthogonal basis of $H_2(\mathbf{CP}^2\# n\overline{\mathbf{C}\mathbf{P}^2};\mathbf{Z})=H_2(\mathbf{CP}^2;\mathbf{Z})\oplus _n H_2(\overline{\mathbf{C}\mathbf{P}^2};\mathbf{Z})$ such that $h^2=1$ and $e_1^2=e_2^2=\dots=e_n^2=-1$. We begin with the proposition below. This proposition is a key of our constructions. 
\begin{proposition}[cf.~\cite{Y2}]
$(1)$ For $a\geq 1$, the complex projective plane $\mathbf{CP}^2$ admits the handle decomposition in Figure~$\ref{fig2}$. \\
$(2)$ For $a\geq 1$, the $4$-manifold $\mathbf{CP}^2\# 2\overline{\mathbf{C}\mathbf{P}^2}$ admits the handle decomposition in Figure~$\ref{fig3}$.
\begin{figure}[ht!]
\begin{center}
\includegraphics[width=3.0in]{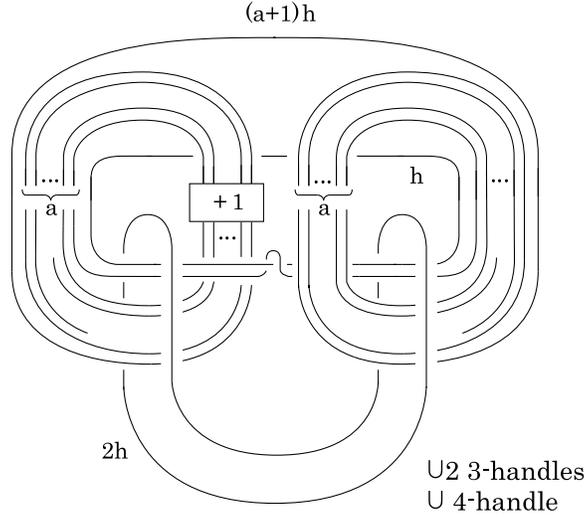}
\caption{$\mathbf{CP}^2$ $(a\geq 1)$}
\label{fig2}
\end{center}
\end{figure}
\begin{figure}[ht!]
\begin{center}
\includegraphics[width=3.0in]{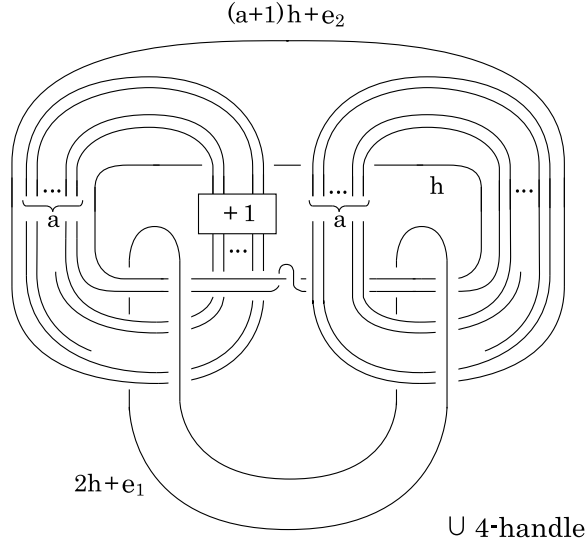}
\caption{$\mathbf{CP}^2\# 2\overline{\mathbf{C}\mathbf{P}^2}$ $(a\geq 1)$}
\label{fig3}
\end{center}
\end{figure}
\end{proposition}
\begin{proof}
(1) Figure~\ref{fig4} is a standard diagram of $\mathbf{CP}^2$. Introduce 2-handle/3-handle pairs and slide handles as in~\cite[Figure $9\sim 14$]{Y2}. Repeat a handle slide as in ~\cite[Figure $15\sim 17$]{Y2}. An isotopy now gives Figure~\ref{fig2}. Notice that the $a=3$ case of Figure~\ref{fig2} is isotopic to ~\cite[Figure 17]{Y2}.\\
(2) Figure~\ref{fig5} is a standard diagram of $\mathbf{CP}^2\# 2\overline{\mathbf{C}\mathbf{P}^2}$. Handle slides similar to the proof of (1) give Figure~\ref{fig3}.
\begin{figure}[h!]
\begin{minipage}{.45\linewidth}
\begin{center}
\includegraphics[width=1.2in]{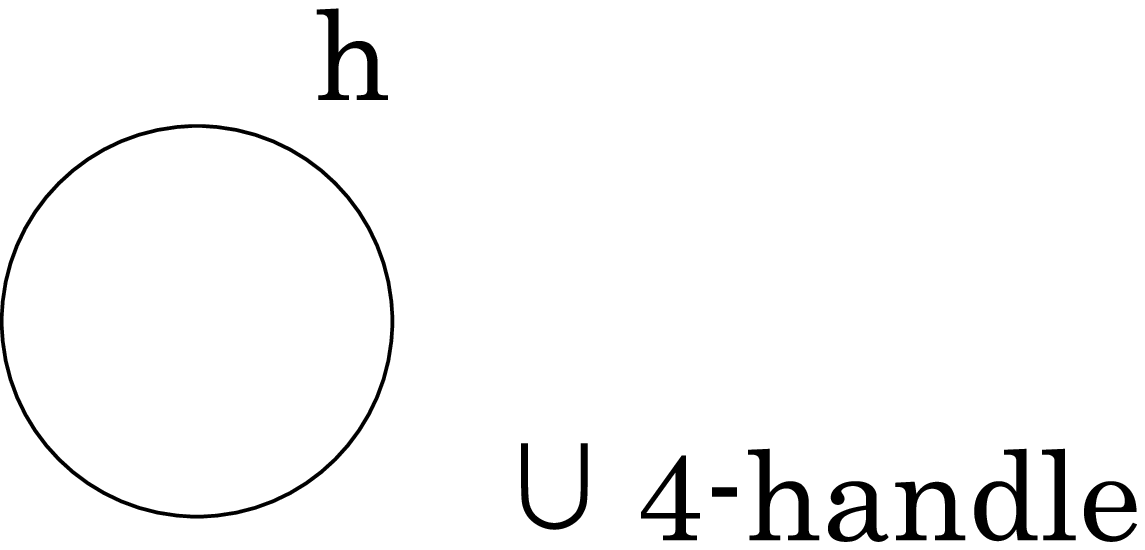}
\caption{$\mathbf{CP}^2$}
\label{fig4}
\end{center}
\end{minipage}
\begin{minipage}{.45\linewidth}
\begin{center}
\includegraphics[width=2.2in]{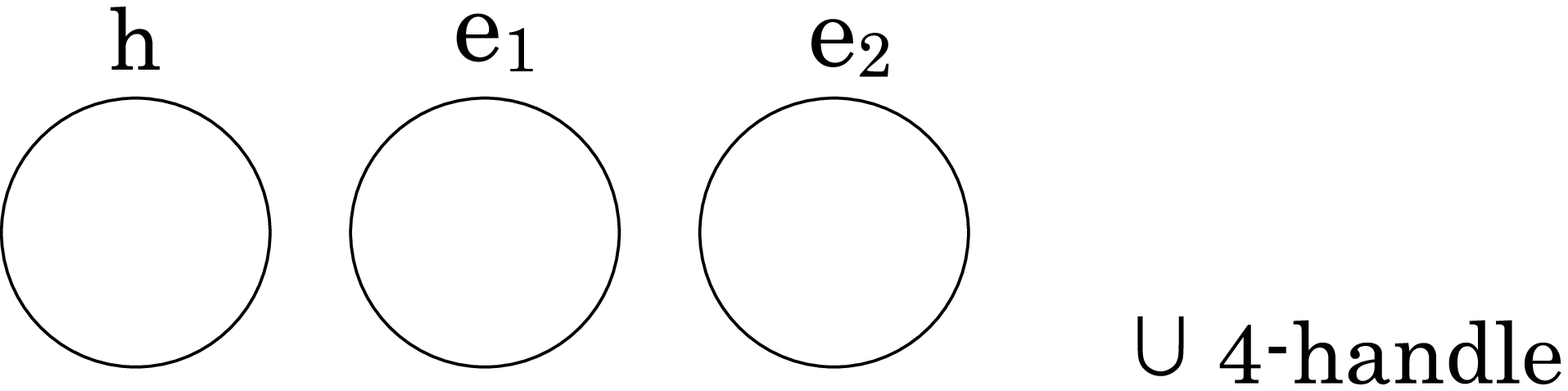}
\caption{$\mathbf{CP}^2\# 2\overline{\mathbf{C}\mathbf{P}^2}$}
\label{fig5}
\end{center}
\end{minipage}
\end{figure}
\end{proof}
\begin{proposition}\label{prop:construction}
$(1)$ $\mathbf{CP}^2\# (3a+2)\overline{\mathbf{C}\mathbf{P}^2}$ $(a\geq 3)$ admits a handle decomposition as in Figure~$\ref{fig6}$.
\begin{figure}[h!]
\begin{center}
\includegraphics[width=3.5in]{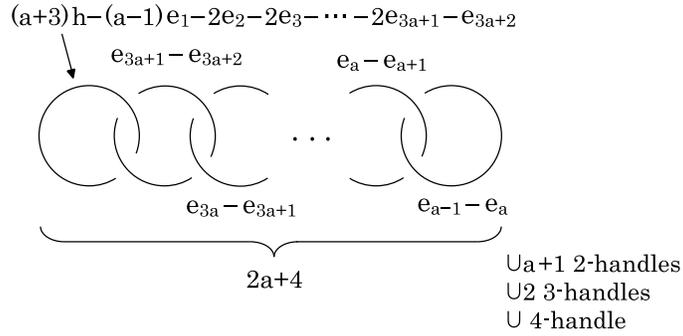}
\caption{$\mathbf{CP}^2\# (3a+2)\overline{\mathbf{C}\mathbf{P}^2}$ $(a\geq 3)$}
\label{fig6}
\end{center}
\end{figure}\\
$(2)$ $\mathbf{CP}^2\# (3a+4)\overline{\mathbf{C}\mathbf{P}^2}$ $(a\geq 3)$ admits a handle decomposition as in Figure~$\ref{fig7}$.
\begin{figure}[h!]
\begin{center}
\includegraphics[width=3.5in]{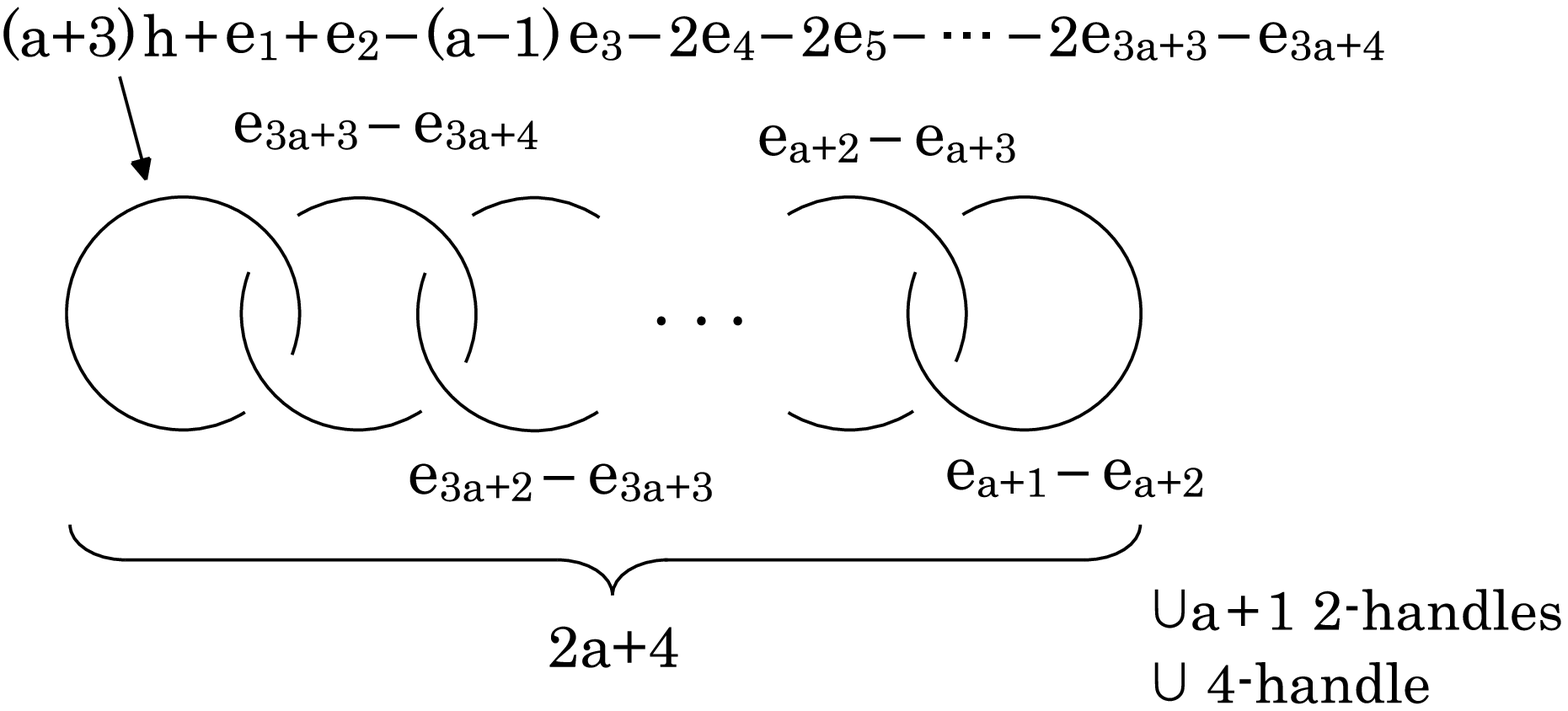}
\caption{$\mathbf{CP}^2\# (3a+4)\overline{\mathbf{C}\mathbf{P}^2}$ $(a\geq 3)$}
\label{fig7}
\end{center}
\end{figure}
\end{proposition}
\begin{proof}
(1) Start with Figure~\ref{fig2}. Blow up as in Figure~\ref{fig8}. An isotopy gives Figure~\ref{fig9}. Isotope and slide a 2-handle as in Figure~\ref{fig10}. Blowing ups make the first diagram of Figure~\ref{fig11}. Handle slides give the second diagram of Figure~\ref{fig11}. We now obtain the last diagram of Figure~\ref{fig11} by blowing up. This diagram clearly provides us Figure~\ref{fig6}.\\
(2) Start with Figure~\ref{fig3}. Blowing ups and handle slides similar to the proof of (1) give Figure~\ref{fig7}.
\begin{figure}[h!]
\begin{center}
\includegraphics[width=4.0in]{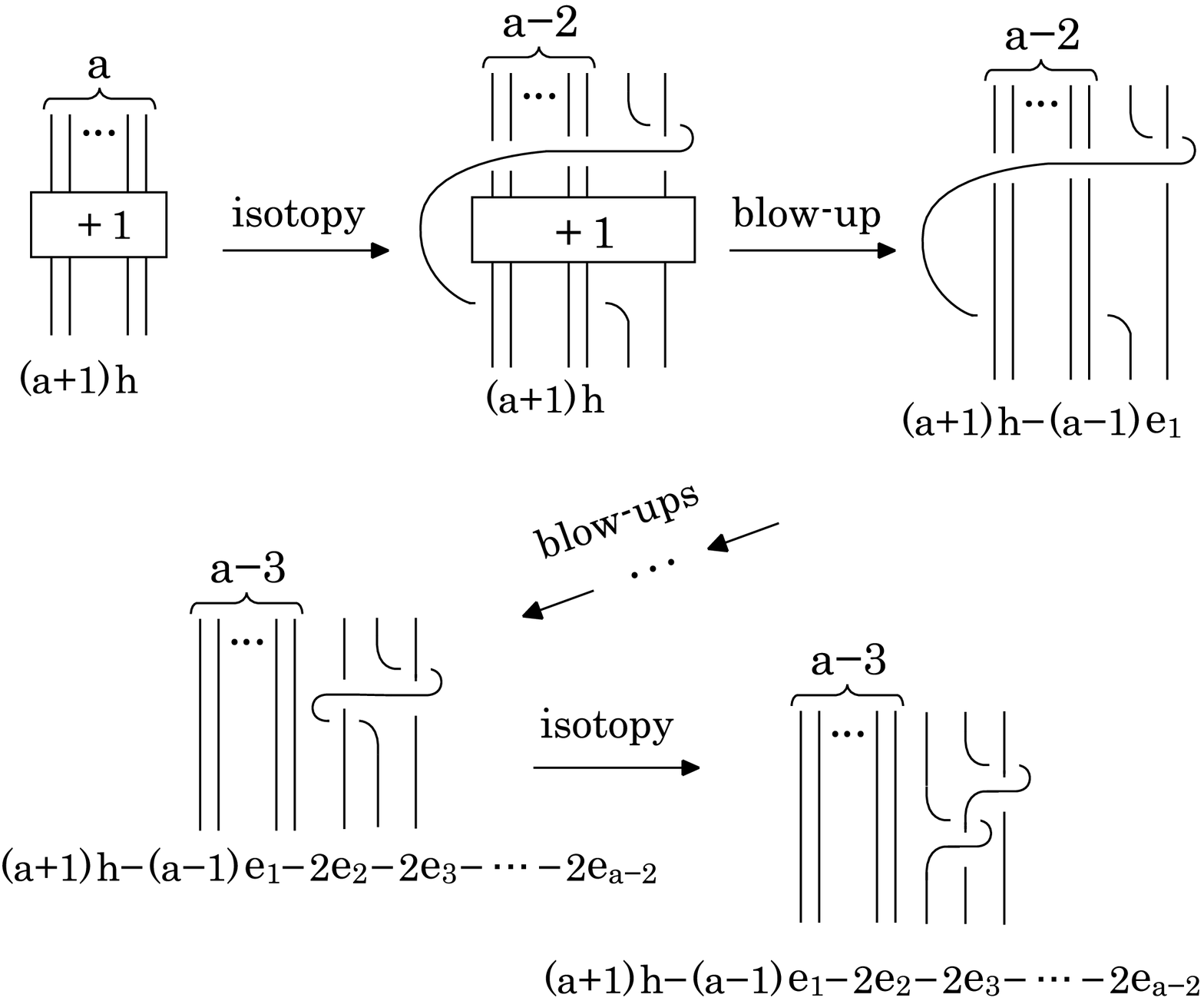}
\caption{blow-ups of $\mathbf{CP}^2$ $(a\geq 3)$}
\label{fig8}
\end{center}
\end{figure}
\begin{figure}[h!]
\begin{center}
\includegraphics[width=2.3in]{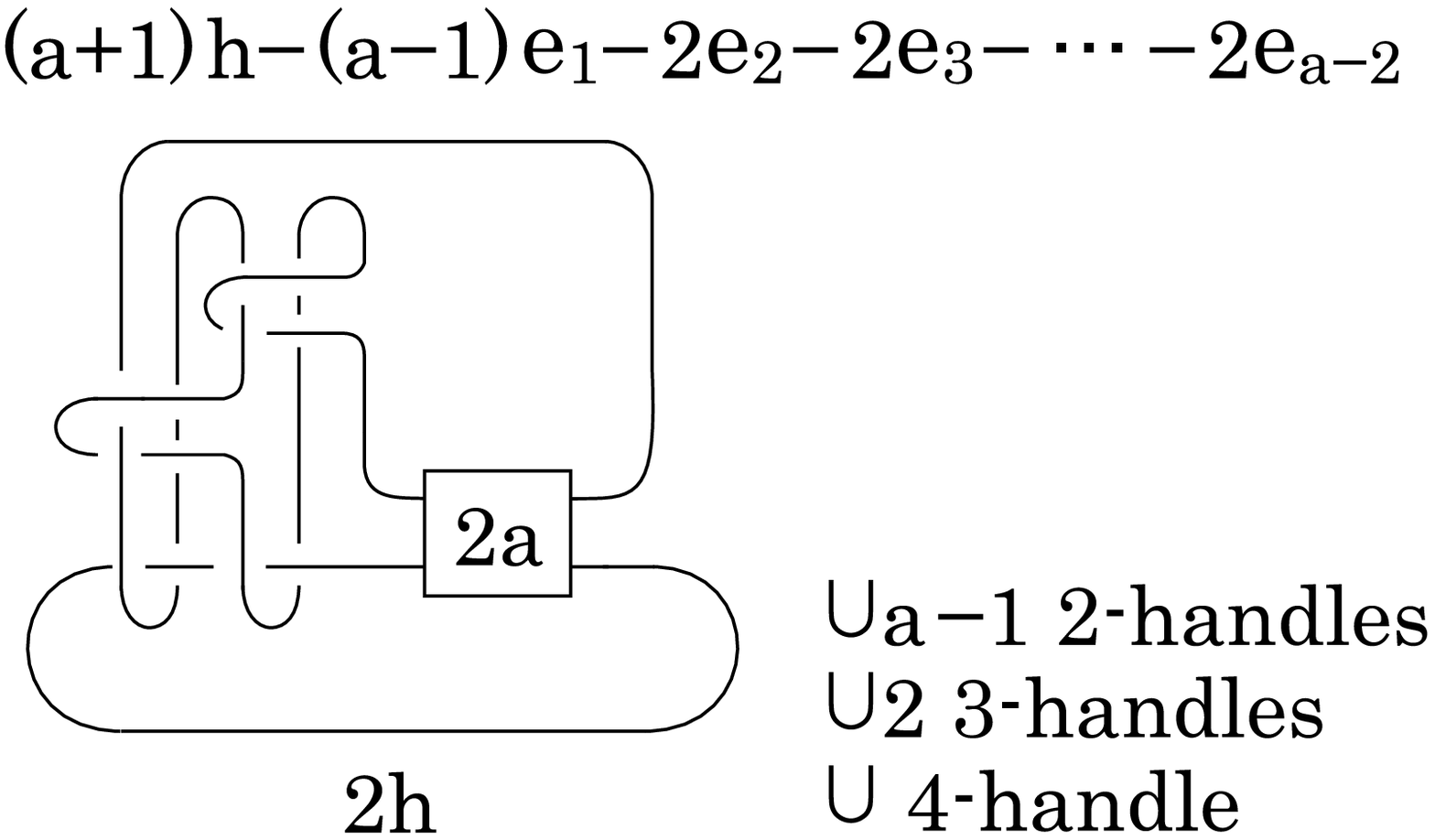}
\caption{$\mathbf{CP}^2\# (a-2)\overline{\mathbf{C}\mathbf{P}^2}$ $(a\geq 3)$}
\label{fig9}
\end{center}
\end{figure}
\begin{figure}[h!]
\begin{center}
\includegraphics[width=4.8in]{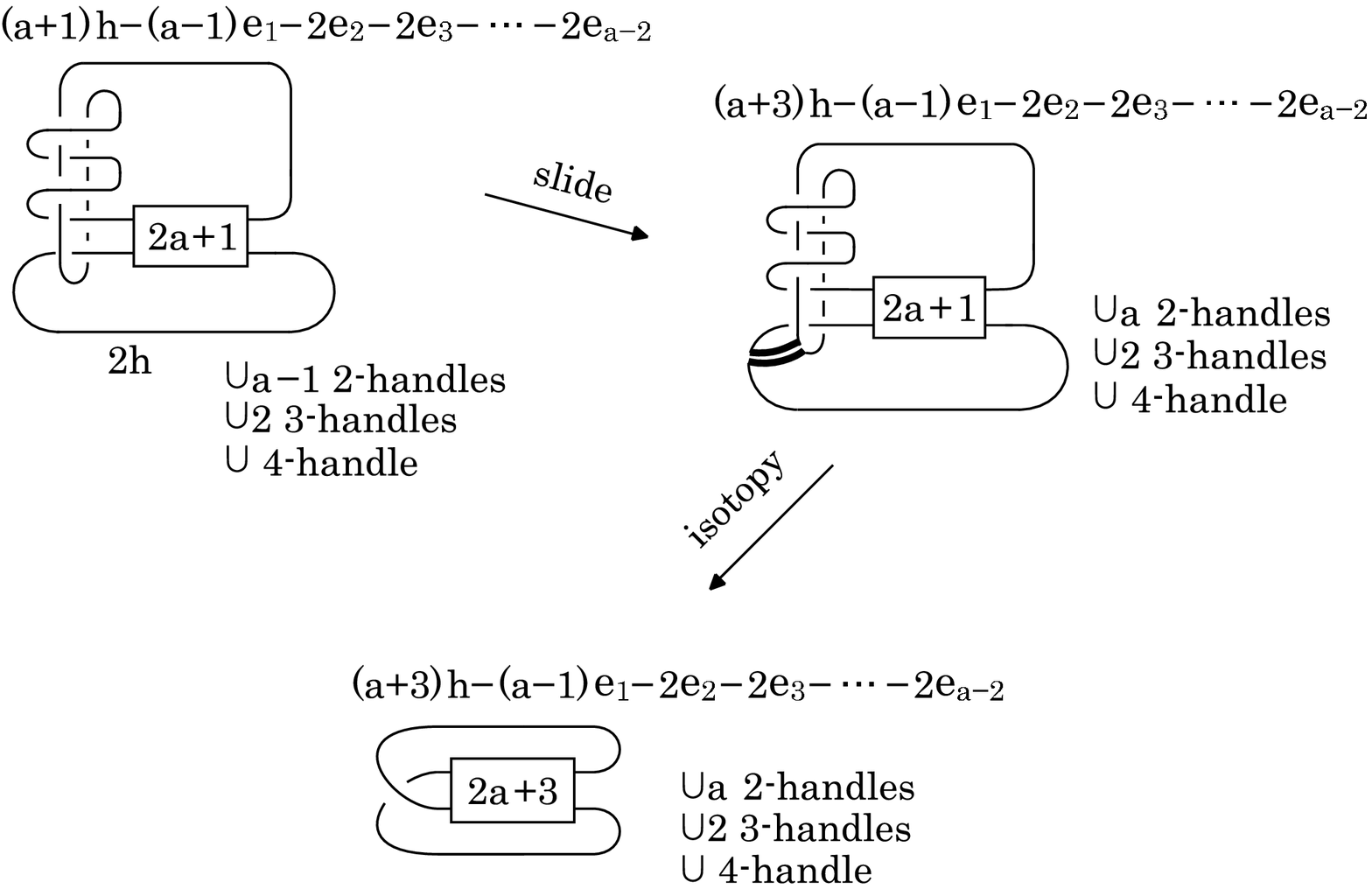}
\caption{$\mathbf{CP}^2\# (a-2)\overline{\mathbf{C}\mathbf{P}^2}$ $(a\geq 3)$}
\label{fig10}
\end{center}
\end{figure}
\begin{figure}[h!]
\begin{center}
\includegraphics[width=4.8in]{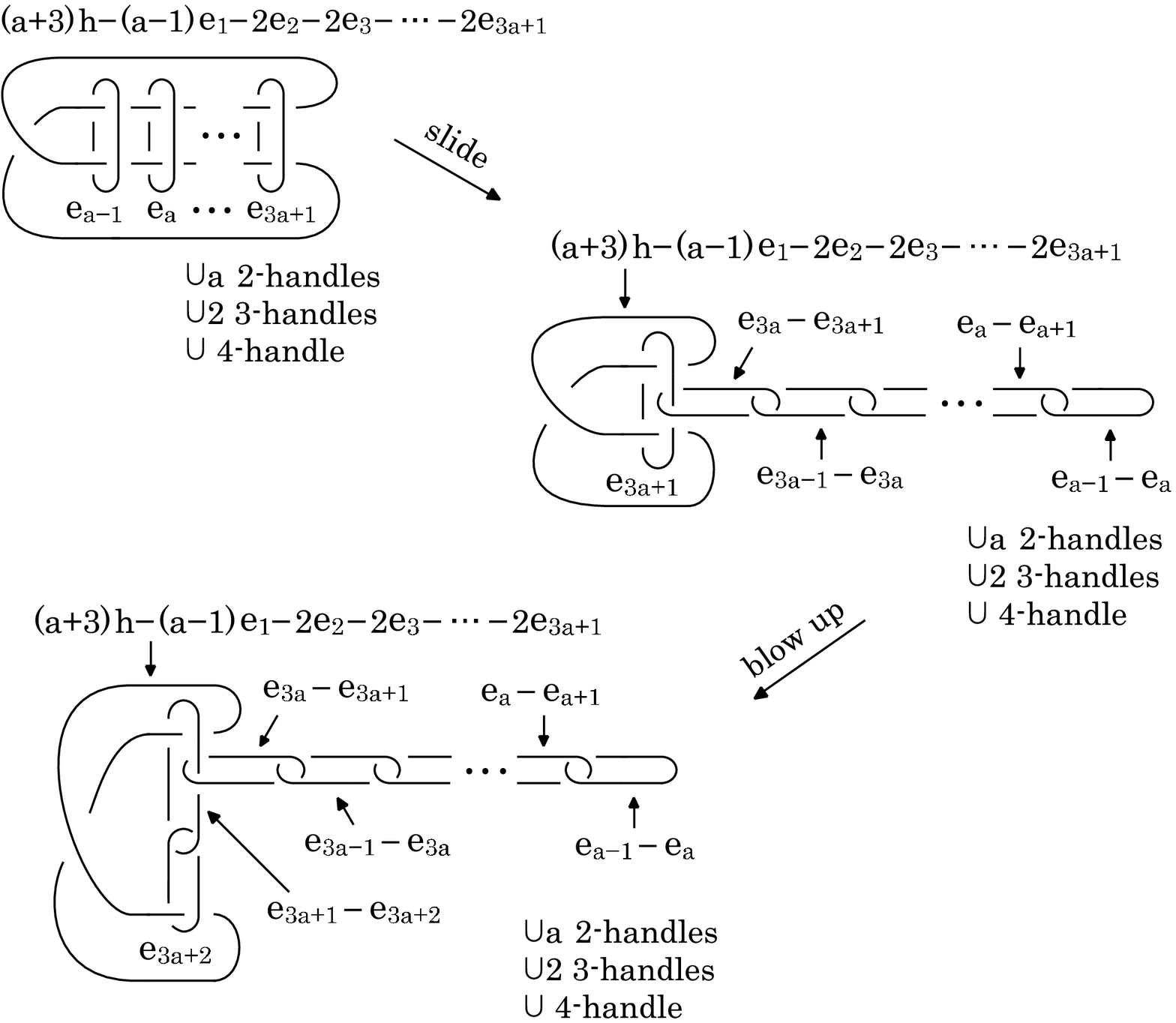}
\caption{handle slides and a blow-up of $\mathbf{CP}^2\# (3a+1)\overline{\mathbf{C}\mathbf{P}^2}$ $(a\geq 3)$}
\label{fig11}
\end{center}
\end{figure}
\end{proof}
We can find $C_p$ in Figure~\ref{fig6} and \ref{fig7}: 
\begin{corollary}\label{cor:C_p}
$(1)$ $\mathbf{CP}^2\# (3a+2)\overline{\mathbf{C}\mathbf{P}^2}$ $(3\leq a\leq 7)$ contains a copy of $C_{4a-9}$ such that the elements $u_1,u_2,\dots,u_{4a-10}$ of $H_2(C_{4a-9};\mathbf{Z})$ in $H_2(\mathbf{CP}^2\# (3a+2)\overline{\mathbf{C}\mathbf{P}^2};\mathbf{Z})$ satisfy
\begin{align*}
 &u_i=e_{12-a+i}-e_{13-a+i}\: (1\leq i\leq 4a-11)\: \text{and}\\
 &u_{4a-10}=(a+3)h-(a-1)e_1-2e_2-2e_3-\dots-2e_{3a+1}-e_{3a+2}. 
\end{align*}
$(2)$ $\mathbf{CP}^2\# (3a+4)\overline{\mathbf{C}\mathbf{P}^2}$ $(3\leq a\leq 6)$ contains a copy of $C_{4a-7}$ such that the elements $u_1,u_2,\dots,u_{4a-8}$ of $H_2(C_{4a-7};\mathbf{Z})$ in $H_2(\mathbf{CP}^2\# (3a+4)\overline{\mathbf{C}\mathbf{P}^2};\mathbf{Z})$ satisfy 
\begin{align*}
 &u_i=e_{12-a+i}-e_{13-a+i}\: (1\leq i\leq 4a-9)\: \text{and}\\
 &u_{4a-8}=(a+3)h+e_1+e_2-(a-1)e_3-2e_4-2e_5-\dots-2e_{3a+3}-e_{3a+4}. 
\end{align*}
\end{corollary}
\begin{remark}
One can easily check that Corollary~\ref{cor:C_p}.(1) (resp.\ (2)) does not hold for $a\geq 12$ (resp.\ $a\geq 10$). However, we do not know if Corollary~\ref{cor:C_p}.(1) (resp.\ (2)) holds for $8\leq a\leq 11$ (resp. $7\leq a\leq 9$). We discuss this question in Section~6.
\end{remark}
\begin{definition}
Let $X_{a,3}$ $(3\leq a\leq 7)$ be the rational blowdown of $\mathbf{CP}^2\# (3a+2)\overline{\mathbf{C}\mathbf{P}^2}$ along the copy of $C_{4a-9}$ in Corollary~\ref{cor:C_p}.(1). Let $X'_{a,3}$ $(3\leq a\leq 6)$ be the rational blowdown of $\mathbf{CP}^2\# (3a+4)\overline{\mathbf{C}\mathbf{P}^2}$ along the copy of $C_{4a-7}$ in Corollary~\ref{cor:C_p}.(2).
\end{definition}
\begin{remark}
The smooth $4$-manifolds $X_{3,3}$ and $X'_{3,3}$ correspond to the smooth $4$-manifolds $E_3$ and $E'_3$ in \cite{Y2}, respectively. 
\end{remark}
We use the lemma below. We can easily prove this lemma, similarly to the proof of~\cite[Lemma~5.1]{Y2}.
\begin{lemma}\label{lem:simply connected}
Let $X$ be a simply connected closed smooth $4$-manifold which contains a copy of $C_p$. Let $X_{(p)}$ be the rational blowdown of $X$ along the copy of $C_p$.\\
$(1)$ Suppose that there exists an element $\delta$ of $H_2(X;\mathbf{Z})$ such that $\delta$ and the elements $u_1,u_2,\dots,u_{p-1}$ of $H_2(C_p;\mathbf{Z})$ in $H_2(X;\mathbf{Z})$ satisfy that $\delta \cdot u_1=1$ and $\delta \cdot u_2=\delta \cdot u_3=\dots=\delta \cdot u_{p-1}=0$. 
Then $H_1(X_{(p)};\mathbf{Z})=0$.\\
$(2)$ Suppose that there exists an element $\delta$ of $H_2(X;\mathbf{Z})$ such that $\delta$ and the elements $u_1,u_2,\dots,u_{p-1}$ of $H_2(C_p;\mathbf{Z})$ in $H_2(X;\mathbf{Z})$ satisfy that $\delta \cdot u_1=\delta \cdot u_2=\dots=\delta \cdot u_{p-2}=0$ and $\gcd (\delta \cdot u_{p-1},\; p)=1$. 
Then $H_1(X_{(p)};\mathbf{Z})=0$.
\end{lemma}
\begin{remark}
We do not use Lemma~\ref{lem:simply connected}.(2), in this section. We use it to prove Proposition~\ref{prop:final exotic} in Section~\ref{section:idea}.
\end{remark}
\begin{proposition}\label{prop:homeo}
The smooth $4$-manifolds $X_{a,3}$ and $X'_{a,3}$ are homeomorphic to $\mathbf{CP}^2\# (12-a)\overline{\mathbf{C}\mathbf{P}^2}$. 
\end{proposition}
\begin{proof} We give a proof for $X_{a,3}$. We can similarly prove for $X'_{a,3}$. 

Recall that the rational homology $4$-ball $B_{4a-9}$ has only one $1$-handle. It thus follows from Proposition~\ref{prop:construction} that $X_{a,3}$ admits a handle decomposition such that the number of $1$-handles is one. Hence the fundamental group of $X_{a,3}$ is commutative. Define an element $\delta$ of $H_2(\mathbf{CP}^2\# (3a+2)\overline{\mathbf{C}\mathbf{P}^2};\mathbf{Z})$ by $\delta:=e_{12-a}-e_{13-a}$. Lemma~\ref{lem:simply connected}.(1) then shows that $H_1(X_{a,3};\mathbf{Z})=0$. Therefore $X_{a,3}$ is simply connected. 

Since $X_{a,3}$ is obtained from $\mathbf{CP}^2\# (3a+2)\overline{\mathbf{C}\mathbf{P}^2}$ by rationally blowing down $C_{4a-7}$, we easily have $b_2^+=1$ and $b_2^-=12-a$. Therefore Freedman's theorem together with Rochlin's theorem shows that $X_{a,3}$ is homeomorphic to $\mathbf{CP}^2\# (12-a)\overline{\mathbf{C}\mathbf{P}^2}$.
\end{proof}
We can easily prove the lemma below, following the rational blowdown procedure introduced by Gompf-Stipsicz~\cite[Section~8.5]{GS}. 
\begin{lemma}[{\cite[Lemma~3.6]{Y2}}]\label{without-handle}Suppose that a simply connected closed smooth 
$4$-manifold $X$ has a handle decomposition 
as in Figure~$\ref{fig12}$. Here $n$ is an arbitrary integer, $h_2$ and $h_3$ are arbitrary natural numbers. Note that we write usual framings instead of homology classes in the figure. 

Let $X_{(p)}$ be the rational blowdown of $X$ along the copy of $C_p$ in Figure~$\ref{fig12}$. 
Then $X_{(p)}$ admits a handle decomposition 
\begin{equation*}
X_{(p)}=\text{one $0$-handle} \cup \text{$(h_2+1)$ $2$-handles} \cup \text{$h_3$ $3$-handles} \cup \text{one $4$-handle}.
\end{equation*}
In particular $X_{(p)}$ admits a handle decomposition without $1$-handles.
\begin{figure}[ht!]
\begin{center}
\includegraphics[width=3.0in]{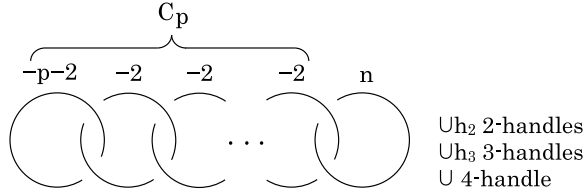}
\caption{Handle decomposition of $X$}
\label{fig12}
\end{center}
\end{figure}
\end{lemma}\medskip 

Lemma~\ref{without-handle} together with Proposition~\ref{prop:construction} gives the following proposition. 
\begin{proposition}\label{prop:handle}
$(1)$ For $3\leq a\leq 6$, the smooth $4$-manifold $X_{a,3}$ admits a handle decomposition without $1$-handles, namely, 
\begin{equation*}
X_{a,3}=\text{one $0$-handle} \cup \text{$(14-a)$ $2$-handles} \cup \text{two $3$-handles} \cup \text{one $4$-handle}.
\end{equation*}
$(2)$ For $3\leq a\leq 5$, the smooth $4$-manifold $X'_{a,3}$ admits a handle decomposition without $1$- and $3$-handles, namely, 
\begin{equation*}
X'_{a,3}=\text{one $0$-handle} \cup \text{$(12-a)$ $2$-handles} \cup \text{one $4$-handle}.
\end{equation*} 
$(3)$ For $a=6$, the smooth $4$-manifold $X'_{a,3}$ admits a handle decomposition without $3$-handles, namely, 
\begin{equation*}
X'_{a,3}=\text{one $0$-handle} \cup \text{one $1$-handle} \cup \text{$(13-a)$ $2$-handles} \cup \text{one $4$-handle}.
\end{equation*} 
\end{proposition}
\section{Seiberg-Witten invariants}
In this section, we briefly review facts about the Seiberg-Witten invariants with $b_2^+=1$. 
For details and examples of computations, see Fintushel-Stern \cite{FS3}, \cite{FS1}, \cite{FS2}, Stern \cite{S}, Park \cite{P1}, \cite{P2}, Ozsv\'{a}th-Szab\'{o} \cite{OS}, Stipsicz-Szab\'{o} \cite{SS} and Park-Stipsicz-Szab\'{o} \cite{PSS}.

Suppose that $X$ is a simply connected closed smooth $4$-manifold with $b_2^+(X)=1$. Let $\mathcal{C}(X)$ be the set of 
characteristic elements of $H^2(X;\mathbf{Z})$. Fix a homology orientation on $X$, that is, 
orient $H^2_+(X;\mathbf{R}):=\{ H\in H^2(X;\mathbf{Z})\, |\, H^2>0\} $. 
Then the (small-perturbation) Seiberg-Witten invariant $SW_{X,H}(K)\in \mathbf{Z}$ is defined for 
every positively oriented element $H\in H^2_+(X;\mathbf{R})$ and every element $K\in \mathcal{C}(X)$ such that $K\cdot H\neq 0$. 
Let $e(X)$ and $\sigma(X)$ be the Euler characteristic and the signature of $X$, respectively, and $d_X(K)$ the even integer defined by 
$d_X(K)=\frac{1}{4}(K^2-2e(X)-3\sigma(X))$ 
for $K\in \mathcal{C}(X)$. It is known that if $SW_{X,H}(K)\neq 0$ for some $H\in H^2_+(X;\mathbf{R})$, then $d_X(K)\ge 0$. 
The wall-crossing formula tells us the dependence of $SW_{X,H}(K)$ on $H$:
if $H, H' \in H^2_+(X;\mathbf{R})$ and $K\in \mathcal{C}(X)$ satisfy $H\cdot H'>0$ and $d_X(K)\ge 0$, then
\begin{multline*}
SW_{X,H'}(K)=SW_{X,H}(K)\\
+
\begin{cases}
0&\text{if $K\cdot H$ and $K\cdot H'$ have the same sign,}\\
(-1)^{\frac{1}{2}d_X(K)}&\text{if $K\cdot H>0$ and $K\cdot H'<0$,}\\
(-1)^{1+\frac{1}{2}d_X(K)}&\text{if $K\cdot H<0$ and $K\cdot H'>0$}.
\end{cases}
\end{multline*}
Note that these facts imply that $SW_{X,H}(K)$ is independent of $H$ in the case $b_2^-(X)\leq 9$, in other words, 
the Seiberg-Witten invariant $SW_{X}:\mathcal{C}(X)\to \mathbf{Z}$ is well-defined.

We recall the change of the Seiberg-Witten invariants by rationally blowing down. Assume that $X$ contains a copy of $C_p$. Let $X_{(p)}$ be the rational blowdown of $X$ along the copy of $C_p$. Suppose that $X_{(p)}$ is simply connected. The following theorems are well-known.
\begin{theorem}[{Fintushel-Stern \cite{FS1}}]\label{thm:4.1}
For every characteristic element $K$ of $H^2(X_{(p)};\mathbf{Z})$, there exists a characteristic element $\tilde{K}$ of $H^2(X;\mathbf{Z})$ such that 
$K\rvert _{X_{(p)}-\text{\normalfont{int}}\,B_{p}}=\tilde{K}\rvert _{X-\text{\normalfont{int}}\,C_{p}}$ and 
$d_{X_{(p)}}(K)=d_X(\tilde{K})$. Such a characteristic element $\tilde{K}$ of $H^2(X;\mathbf{Z})$ is called a lift of $K$.
\end{theorem}
\begin{theorem}[{Fintushel-Stern \cite{FS1}}]\label{nthm:4.1}
If a characteristic element $\tilde{K}$ of $H^2(X;\mathbf{Z})$ is a lift of some characteristic element $K$ of $H^2(X_{(p)};\mathbf{Z})$, then $SW_{X_{(p)},H}(K)=SW_{X,H}(\tilde{K})$ for every positively oriented element $H$ of $H^2_+(X;\mathbf{R})$ 
which is orthogonal to the subspace $H_2(C_p;\mathbf{R})$ of $H_2(X;\mathbf{R})$. Note that we view $H$ as a positively oriented element of $H^2_+(X_{(p)};\mathbf{R})$.
\end{theorem}
\begin{theorem}[{Fintushel-Stern \cite{FS1}, cf.~Park \cite{P1}}]\label{thm:4.2}
If a characteristic element $\tilde{K}$ of $H^2(X;\mathbf{Z})$ satisfies that $(\tilde{K}\rvert _{C_p})^2=1-p$ and 
$\tilde{K}\rvert _{\partial C_p}=mp\in \mathbf{Z}_{p^2}\cong H^2(\partial C_p;\mathbf{Z})$ 
with $m\equiv p-1\pmod 2$, then there exists a characteristic element $K$ of $H^2(X_{(p)};\mathbf{Z})$ such that $\tilde{K}$ is a lift of $K$. 
\end{theorem}
\begin{corollary}\label{cor:4.3}
If a characteristic element $\tilde{K}$ of $H^2(X;\mathbf{Z})$ satisfies $\tilde{K}(u_1)=\tilde{K}(u_2)=\dots=\tilde{K}(u_{p-2})=0$ and $\tilde{K}(u_{p-1})=\pm p$, then $\tilde{K}$ is a lift of some characteristic element $K$ of $H^2(X_{(p)};\mathbf{Z})$.
\end{corollary}
\section{Computations of SW invariants}
In this section, we prove our manifolds are homeomorphic but not diffeomorphic to rational surfaces. We denote the symbol $R_n$ as $\mathbf{CP}^2\# n\overline{\mathbf{C}\mathbf{P}^2}$. 
\begin{lemma}\label{lem:5.1}
$(1)$ For $3\leq a\leq 7$, there exists a characteristic element $K_{a,3}$ of $H^2(X_{a,3};\mathbf{Z})$ such that $SW_{X_{a,3}}(\pm K_{a,3})=\pm 1$. \\
$(2)$ For $3\leq a\leq 6$, there exists a characteristic element $K'_{a,3}$ of $H^2(X'_{a,3};\mathbf{Z})$ such that $SW_{X'_{a,3}}(\pm K'_{a,3})=\pm 1$. 
\end{lemma}
\begin{proof}
(1) Let $\tilde{K}_{a,3}\in H^2(R_{3a+2};\mathbf{Z})$ and $H_{a,3}\in H_+^2(R_{3a+2};\mathbf{R})$ be the elements defined by $\tilde{K}_{a,3}=PD(3h-e_1-e_2-\dots-e_{3a+2})$ and $H_{a,3}=PD((8a-1)h-2(a+3)e_1-(a+3)e_2-(a+3)e_3-\dots-(a+3)e_{3a+2})$. Here the symbol $PD$ denotes the Poincar\'{e} dual. It is well known that $SW_{R_{n},PD(h)}(\tilde{K})=0$ for every $n\geq 0$ and every characteristic element $\tilde{K}$ of $H^2(R_n;\mathbf{Z})$. 
Applying the wall-crossing formula to $\pm \tilde{K}_{a,3}$, $H_{a,3}$ and $PD(h)$, we have $SW_{R_{3a+2},H_{a,3}}(\pm \tilde{K}_{a,3})=\pm 1$. 
Corollary~\ref{cor:4.3} shows that $\tilde{K}_{a,3}$ is a lift of some characteristic element $K_{a,3}$ of $H^2(X_{a,3};\mathbf{Z})$. Since $H_{a,3}$ is orthogonal to to the subspace $H_2(C_{4a-9};\mathbf{R})$ of $H_2(R_{3a+2};\mathbf{R})$, Theorem~\ref{nthm:4.1} shows $SW_{X_{a,3}}(\pm K_{a,3})=SW_{R_{3a+2},H_{a,3}}(\pm \tilde{K}_{a,3})=\pm 1$. 

(2) Let $\tilde{K}'_{a,3}\in H^2(R_{3a+2};\mathbf{Z})$ and $H'_{a,3}\in H_+^2(R_{3a+2};\mathbf{R})$ be the elements defined by $\tilde{K}_{a,3}=PD(3h+e_1+e_2-e_3-e_4-\dots-e_{3a+4})$ and $H'_{a,3}=PD((8a+1)h+(a+3)e_1+(a+3)e_2-2(a+3)e_3-(a+3)e_4-(a+3)e_5-\dots-(a+3)e_{3a+4})$. We now can prove the required claim, similarly to the proof of (1). 
\end{proof}
The corollary below follows from the fact $SW_{\mathbf{CP}^2\# n\overline{\mathbf{C}\mathbf{P}^2}}\equiv 0$ $(0\leq n\leq 9)$. 
\begin{corollary}
$(1)$ For $3\leq a\leq 7$, the $4$-manifold $X_{a,3}$ is not diffeomorphic to $\mathbf{CP}^2\# (12-a)\overline{\mathbf{C}\mathbf{P}^2}$. \\
$(2)$ For $3\leq a\leq 6$, the $4$-manifold $X'_{a,3}$ is not diffeomorphic to $\mathbf{CP}^2\# (12-a)\overline{\mathbf{C}\mathbf{P}^2}$. 
\end{corollary}
This corollary together with Propositions~\ref{prop:homeo} and \ref{prop:handle} shows the following main theorem. 
\begin{theorem}
$(1)$ For $3\leq a\leq 7$, the smooth $4$-manifold $X_{a,3}$ is homeomorphic but not diffeomorphic to $\mathbf{CP}^2\# (12-a)\overline{\mathbf{C}\mathbf{P}^2}$. Furthermore, $X_{a,3}$ $(3\leq a\leq 6)$ admits a handle decomposition without $1$-handles, namely, 
\begin{equation*}
X_{a,3}=\text{one $0$-handle} \cup \text{$(14-a)$ $2$-handles} \cup \text{two $3$-handles} \cup \text{one $4$-handle}.
\end{equation*}
$(2)$ For $3\leq a\leq 6$, the smooth $4$-manifold $X'_{a,3}$ is homeomorphic but not diffeomorphic to $\mathbf{CP}^2\# (12-a)\overline{\mathbf{C}\mathbf{P}^2}$. Furthermore, 

$(\textnormal{i})$ $X'_{a,3}$ $(3\leq a\leq 5)$ admits a handle decomposition without $1$- and $3$-handles, namely, 
\begin{equation*}
X'_{a,3}=\text{one $0$-handle} \cup \text{$(12-a)$ $2$-handles} \cup \text{one $4$-handle};
\end{equation*}

$(\textnormal{ii})$ $X'_{a,3}$ $(a=6)$ admits a handle decomposition without $3$-handles, namely, 
\begin{equation*}
X'_{a,3}=\text{one $0$-handle} \cup \text{one $1$-handle} \cup \text{$(13-a)$ $2$-handles} \cup \text{one $4$-handle}. 
\end{equation*}
\end{theorem}
\begin{remark}
(1) We can prove similarly to \cite{Y2} that the Seiberg-Witten basic classes of $X_{a,3}$ (resp.\ $X'_{a,3}$) are only $\pm K_{a,3}$ (resp.\ $\pm K'_{a,3}$), at least for $a=3,4$. Therefore $X_{a,3}$ and $X'_{a,3}$ $(a=3,4)$ are minimal 4-manifolds, that is, they contain no smooth 2-sphere with self intersection number $-1$. It seems that they are minimal for every $a$. Perhaps, one may show the minimality by the method in Ozsv\'{a}th-Szab\'{o}~\cite{OS}. \\
(2) Constructions of exotic rational surfaces in \cite{Y2} and this paper are suitable for finding corks. See \cite{AY}. 
\end{remark}
\section{Strategy toward exotic $\mathbf{CP}^2\# \overline{\mathbf{C}\mathbf{P}^2}$}\label{section:idea}
In this section, we propose a strategy for rational blowdown constructions of exotic $\mathbf{CP}^2\# n\overline{\mathbf{C}\mathbf{P}^2}$ $(1\leq n\leq 9)$, though the author could not carry out the strategy for $1\leq n\leq 4$. Our constructions of exotic rational surfaces in \cite{Y2} and Section~\ref{section:construction} are based on the strategy. 
We begin with the question below. 
\begin{question}\label{question:handle}
Does $\mathbf{CP}^2\# 2\overline{\mathbf{C}\mathbf{P}^2}$ admit a handle decomposition as in Figure~\ref{fig13}, for some $a,b\geq 3$ ?
\begin{figure}[h]
\begin{center}
\includegraphics[width=2.2in]{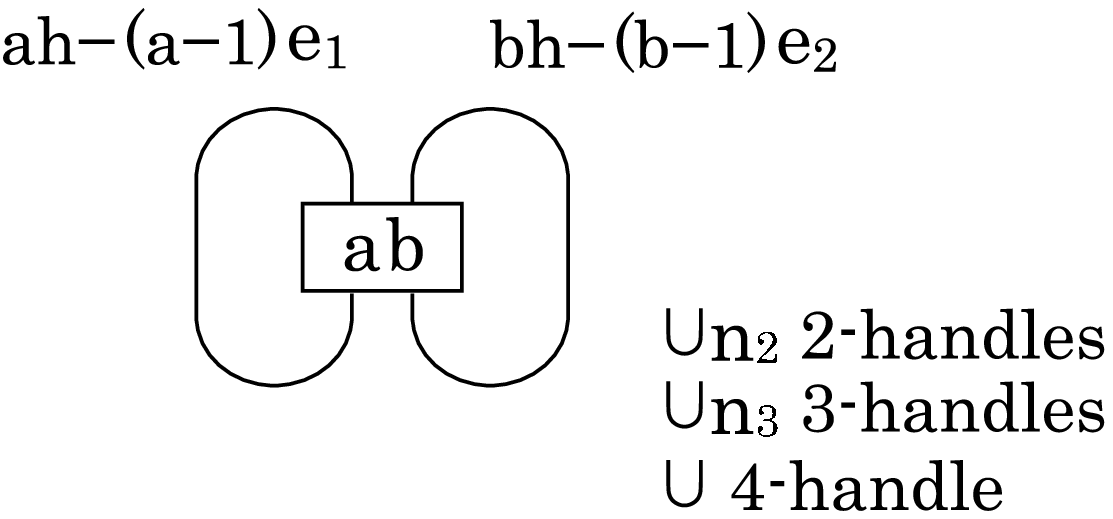}
\caption{}
\label{fig13}
\end{center}
\end{figure}

\end{question}

\begin{remark}(1) Question~\ref{question:handle} seems to be natural, because the elements $ah-(a-1)e_1$ and $bh-(b-1)e_2$ of $H_2(\mathbf{CP}^2\# 2\overline{\mathbf{C}\mathbf{P}^2};\mathbf{Z})$ are represesnted by smooth $2$-spheres in $\mathbf{CP}^2\# 2\overline{\mathbf{C}\mathbf{P}^2}$.\\
(2) One can easily check that Question~\ref{question:handle} is affirmative when $a=1$ and $b$ is an arbitraly positive number. We can also give an affirmative solution when $a=2$ and $b$ is an arbitraly positive number, by blowing up in Figure~\ref{fig2}.
\end{remark}

For some numbers $a$ and $b$, an affirmative solution to Question~\ref{question:handle} gives a construction of exotic $\mathbf{CP}^2\# n\overline{\mathbf{C}\mathbf{P}^2}$ $(2\leq n\leq 9)$ as follows. 

\begin{proposition}\label{prop:configuration from handle}
Suppose that Question~$\ref{question:handle}$ is affirmative. Then $\mathbf{CP}^2\# (ab+2)\overline{\mathbf{C}\mathbf{P}^2}$ admits a handle decomposition as in Figure~$\ref{fig14}$. 
\begin{figure}[h!]
\begin{center}
\includegraphics[width=3.7in]{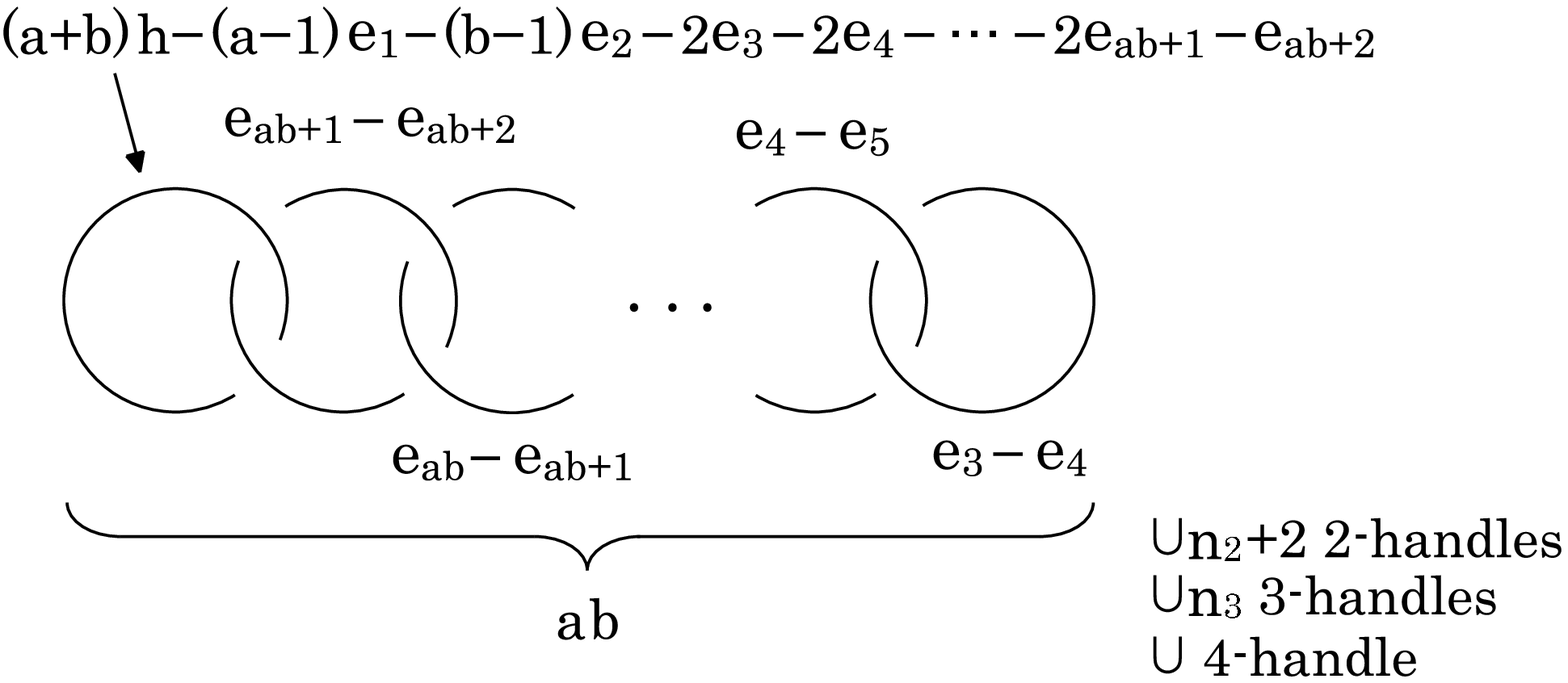}
\caption{}
\label{fig14}
\end{center}
\end{figure}
\end{proposition}
\begin{proof}
Slide handles and blow up as in Figure~\ref{fig17}.
\end{proof}
\begin{question}\label{question:C_p}
$(1)$ Does $\mathbf{CP}^2\# (3a+2)\overline{\mathbf{C}\mathbf{P}^2}$ $(3\leq a \leq 11)$ admit a handle decomposition as in Figure~\ref{fig15}?
\begin{figure}[h!]
\begin{center}
\includegraphics[width=3.5in]{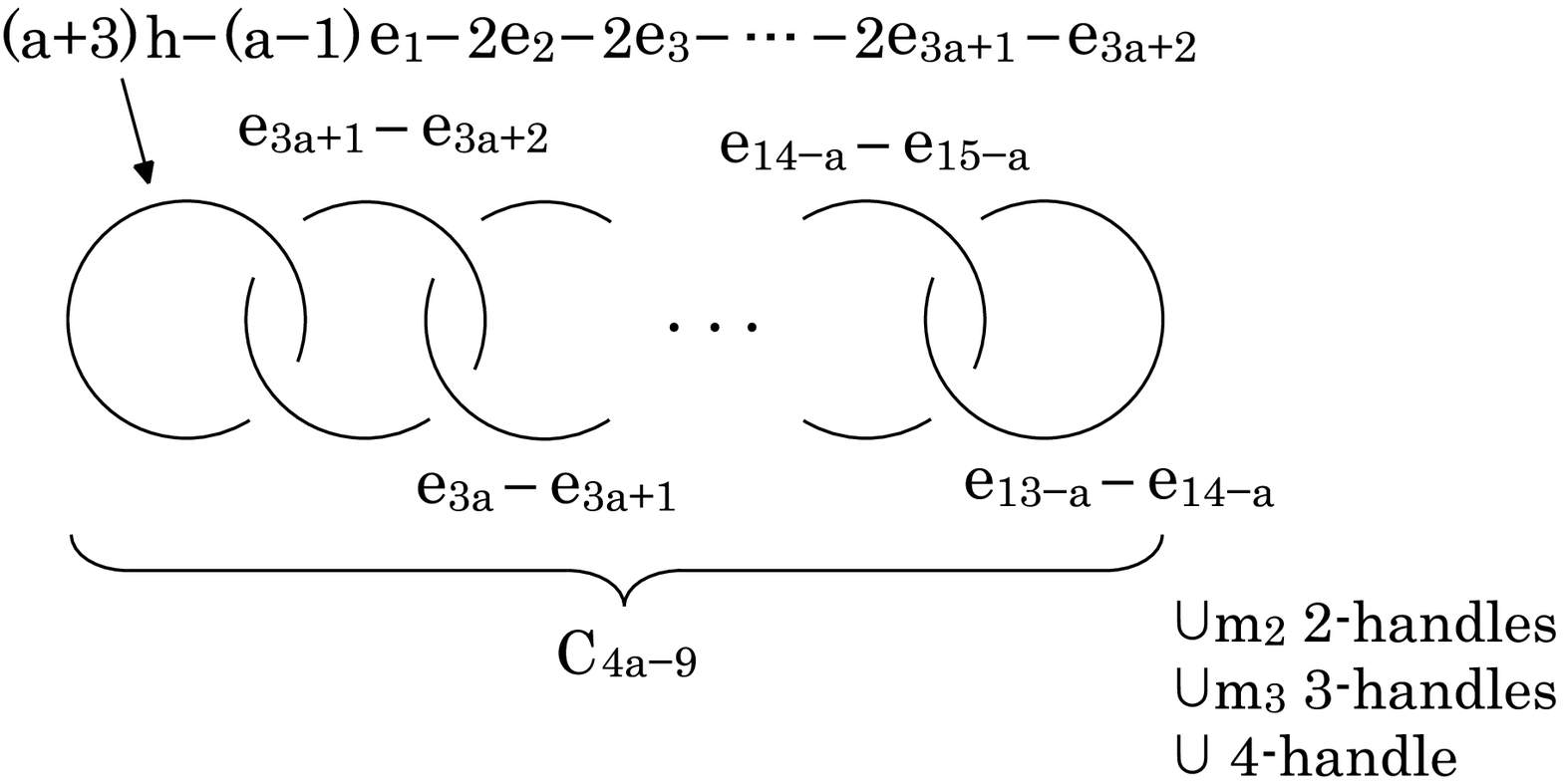}
\caption{}
\label{fig15}
\end{center}
\end{figure}\\
$(2)$ Does $\mathbf{CP}^2\# (4a+2)\overline{\mathbf{C}\mathbf{P}^2}$ $(3\leq a \leq 6)$ admit a handle decomposition as in Figure~\ref{fig16}?
\begin{figure}[h!]
\begin{center}
\includegraphics[width=3.5in]{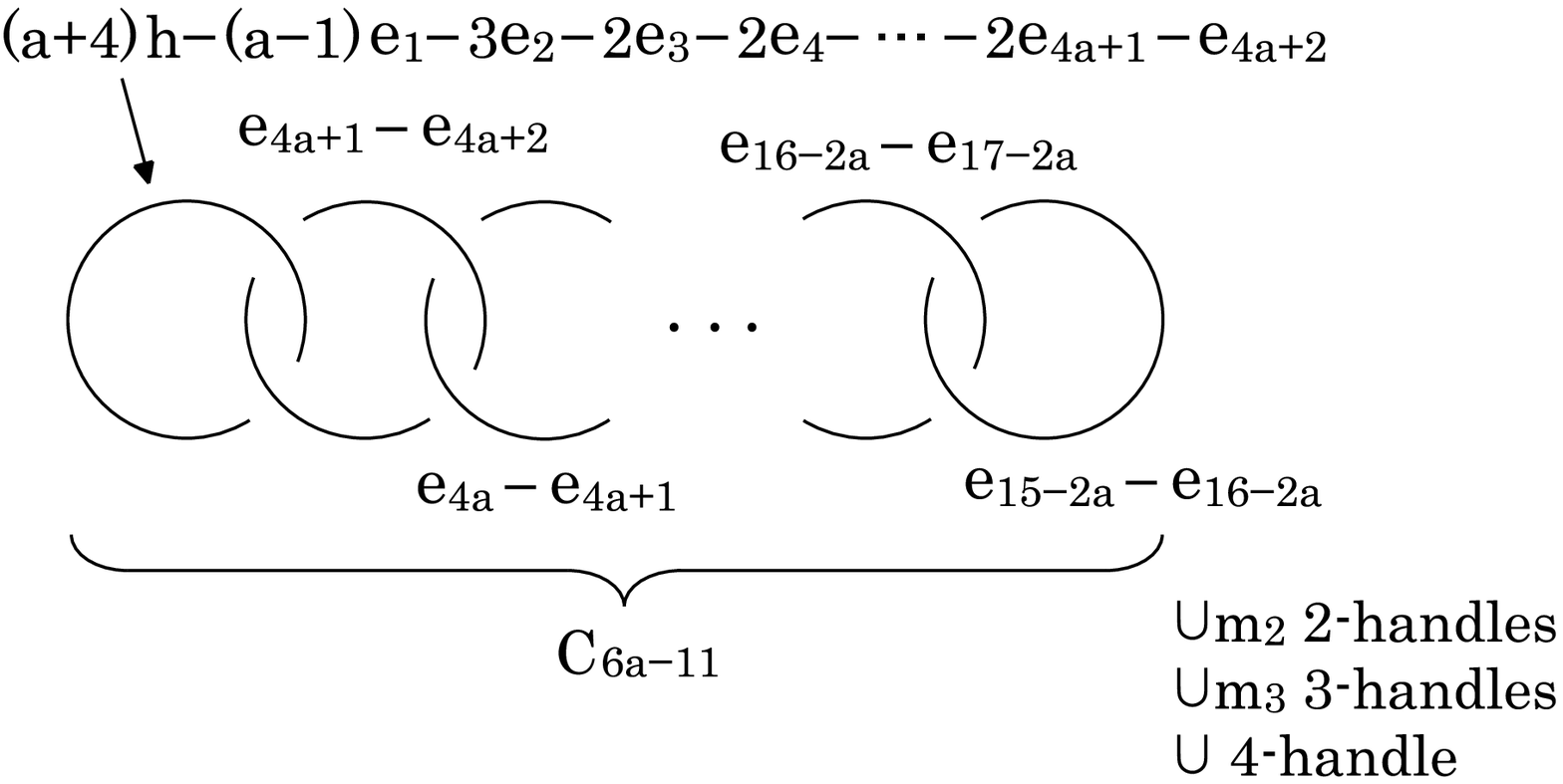}
\caption{}
\label{fig16}
\end{center}
\end{figure}
\end{question}
\begin{remark}
$(1)$ Proposition~\ref{prop:configuration from handle} shows that Question~\ref{question:C_p}.(1) (resp.\ (2)) is affirmative for $3\leq a\leq 10$ (resp.\ $3\leq a\leq 6$) if Question~\ref{question:handle} is affirmative for corresponding $a$ and $b$. \\
$(2)$ In Section~\ref{section:construction}, we proved that Question~\ref{question:C_p}.(1) is affirmative for $3\leq a\leq 7$. We can also prove that Question~\ref{question:C_p}.(2) is affirmative for $3\leq a\leq 4$, by appropriately changing the constructions in Section~\ref{section:construction}. 
\end{remark}
\begin{definition}Suppose that Question~$\ref{question:C_p}$ is affirmative. Let $X_{a,3}$ $(3\leq a\leq 11)$ be the rational blowdown of $\mathbf{CP}^2\# (3a+2)\overline{\mathbf{C}\mathbf{P}^2}$ along the copy of $C_{4a-9}$ in Figure~\ref{fig15}. Let $X_{a,4}$ $(3\leq a\leq 6)$ be the rational blowdown of $\mathbf{CP}^2\# (4a+2)\overline{\mathbf{C}\mathbf{P}^2}$ along the copy of $C_{6a-11}$ in Figure~\ref{fig16}. 
\end{definition}
\begin{proposition}\label{prop:final exotic}
Suppose that Question~$\ref{question:C_p}$ is affirmative. Then, \\
$(1)$ $X_{a,3}$ $(3\leq a \leq 11)$ is homeomorphic but not diffeomorphic to $\mathbf{CP}^2\# (12-a)\overline{\mathbf{C}\mathbf{P}^2}\:;$\\
$(2)$ $X_{a,4}$ $(3\leq a \leq 6)$ is homeomorphic but not diffeomorphic to $\mathbf{CP}^2\# (14-2a)\overline{\mathbf{C}\mathbf{P}^2}$. 
\end{proposition}
\begin{proof}
Lemma~\ref{lem:simply connected} implies that $X_{a,3}$ and $X_{a,4}$ are simply connected. One can easily check that $b^+_2(X_{a,3})=b^+_2(X_{a,4})=1$, $b^-_2(X_{a,3})=12-a$ and $b^-_2(X_{a,4})=14-2a$. Rochlin's theorem shows that the intersection forms of $X_{a,3}$ $(a\neq 11)$ and $X_{a,4}$ are odd. We can also prove that the intersection form of $X_{a,3}$ $(a=11)$ is odd as follows. Let $v$ be the element of $H_2(\mathbf{CP}^2\# (3a+2)\overline{\mathbf{C}\mathbf{P}^2};\mathbf{Z})$ defined by $v=(8a-1)h-2(a+3)e_1-(a+3)e_2-(a+3)e_3-\dots-(a+3)e_{3a+2}$. Then $v$ is orthogonal to the subspace $H_2(C_{4a-9};\mathbf{Z})$ of $H_2(\mathbf{CP}^2\# (3a+2)\overline{\mathbf{C}\mathbf{P}^2};\mathbf{Z})$. We can, similarly to \cite[Remark~5.2.(1)]{Y2}, view $v$ as an element of $H_2(\mathbf{CP}^2\# (3a+2)\overline{\mathbf{C}\mathbf{P}^2}-C_{4a-9};\mathbf{Z})\subset H_2(X_{a,3};\mathbf{Z})$. Since $v^2=121$ when $a=11$, the intersection form of $X_{11,3}$ is odd. Freedman's theorem thus shows that $X_{a,3}$ (resp.\ $X_{a,4}$) is homeomorphic to $\mathbf{CP}^2\# (12-a)\overline{\mathbf{C}\mathbf{P}^2}$ (resp.\ $\mathbf{CP}^2\# (14-2a)\overline{\mathbf{C}\mathbf{P}^2}$). 

We can easily prove that $X_{a,3}$ and $X_{a,4}$ are not diffeomorphic to rational surfaces by computing the Seiberg-Witten invariants, similarly to the proof of Lemma~\ref{lem:5.1}.(1). 
\end{proof}
\begin{remark}
(1) In Question~\ref{question:handle} and \ref{question:C_p}, we assumed that there is no 1-handle in the handle decompositions. This is because we used this assumption to prove that $X_{a,3}$ and $X_{a,4}$ are simply connected. However, the author does not know if we need this assumption. \\
(2) In Section~\ref{section:construction}, we constructed $X_{a,3}$ according to the strategy in this section, though we could not solve Question~\ref{question:handle} for corresponding $a$ and $b$. By modifying the construction of $X_{a,3}$, we constructed $X'_{a,3}$ so that $X'_{a,3}$ naturally has no 3-handles. We can also construct an exotic $\mathbf{CP}^2\# 6\overline{\mathbf{C}\mathbf{P}^2}$ which corresponds to $X_{4,4}$.
\end{remark}
\begin{figure}[h!]
\begin{center}
\includegraphics[width=4.7in]{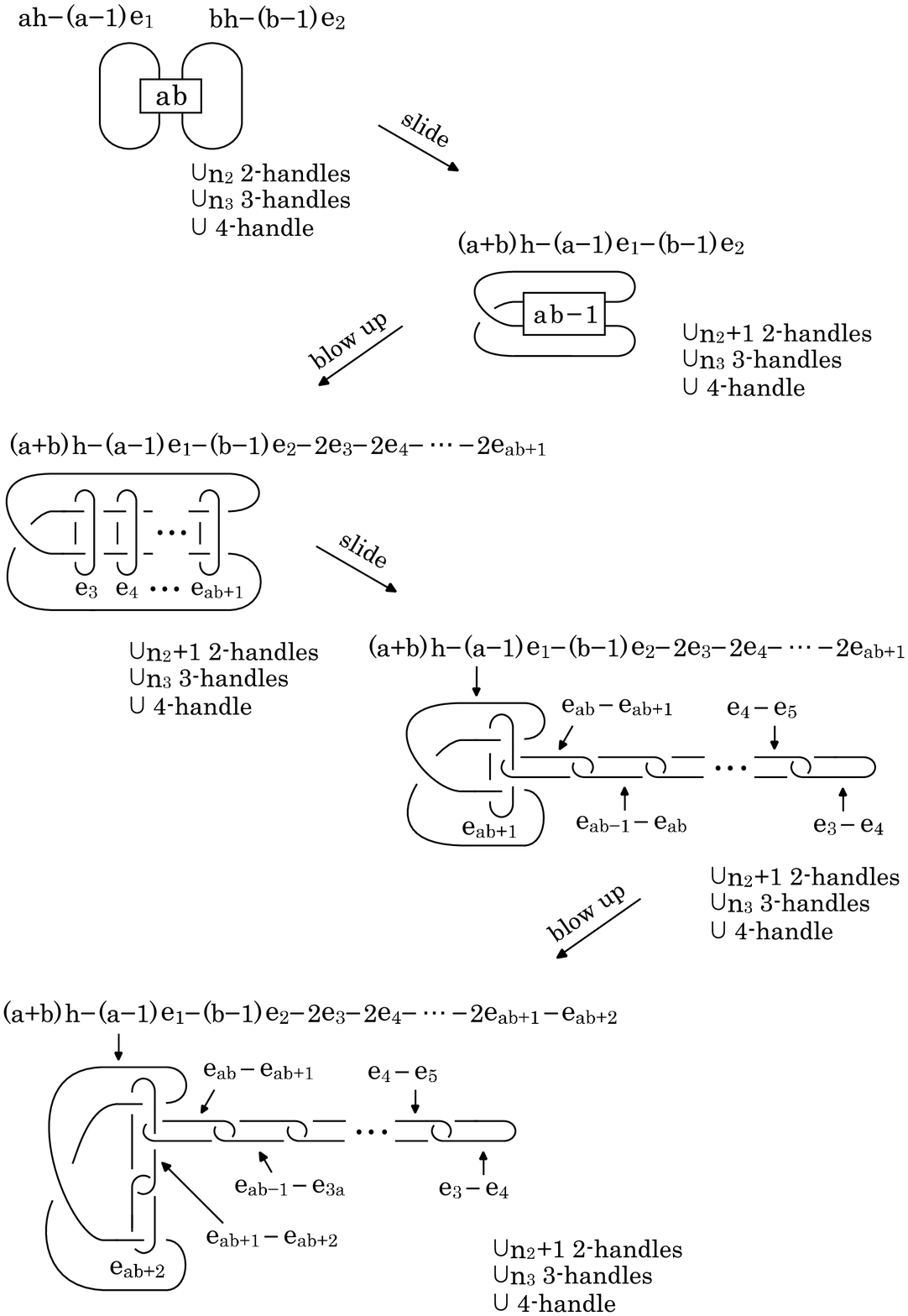}
\caption{}
\label{fig17}
\end{center}
\end{figure}

\end{document}